\documentclass[a4paper,reqno]{amsart}
\usepackage[utf8]{inputenc}
\usepackage{graphicx}
\usepackage{amsmath}
\usepackage{amssymb,amsthm}
\usepackage{mathtools}
\usepackage{etoolbox}
\usepackage[T1]{fontenc}
\usepackage{pgfplots}
\usepackage{hyperref,enumerate}
\usepackage{enumitem}

\setlist[enumerate]{label=\alph*),leftmargin=2\parindent}

\theoremstyle{plain}
\newtheorem{theorem}{Theorem}
\numberwithin{theorem}{section}
\newtheorem{prop}[theorem]{Proposition}

\theoremstyle{definition}

\allowdisplaybreaks

\newcommand{\Z}{\mathbb{Z}}

\newcommand{\R}{\mathbb{R}}
\newcommand{\C}{\mathbb{C}}
\newcommand{\T}{\mathbb{T}}
\newcommand{\RT}{{\R\times\T}}

\renewcommand{\phi}{\varphi}
\renewcommand{\epsilon}{\varepsilon}
\newcommand{\eps}{\epsilon}

\DeclarePairedDelimiterX\norm[1]\lVert\rVert{\ifblank{#1}{\cdot}{#1}}
\DeclarePairedDelimiterX\abs[1]\lvert\rvert{\ifblank{#1}{\cdot}{#1}}

\DeclarePairedDelimiter\set\{\}

\DeclarePairedDelimiterX\intcc[1][]{\ifblank{#1}{0,1}{#1}}
\DeclarePairedDelimiterX\intco[1][){\ifblank{#1}{0,\infty}{#1}}
\DeclarePairedDelimiterX\intoc[1](]{\ifblank{#1}{-\infty,0}{#1}}
\DeclarePairedDelimiterX\intoo[1](){\ifblank{#1}{0,\infty}{#1}}

\newcommand{\Fcal}{\mathcal{F}}

\newcommand{\Ucal}{\mathcal{U}}

\newcommand{\ew}{\newpage\noindent}

\newcommand{\les}{\lesssim}

\begin{document}

\title[improved LWP for quartic ZK on $\RT$]{A note on the well-posedness of the quartic Zakharov-Kuznetsov equation on $\RT$}
\author{Jakob Nowicki-Koth}

\address{J.~Nowicki-Koth: Mathematisches Institut der
Heinrich-Heine-Universit{\"a}t D{\"u}sseldorf, Universit{\"a}tsstr. 1,
40225 D{\"u}sseldorf, Germany}
\email{jakob.nowicki-koth@hhu.de}

\begin{abstract}
By using a bilinear smoothing estimate recently developed in \cite{JNK2025}, together with several linear Strichartz-type estimates established therein, we improve the threshold for local well-posedness of the quartic Zakharov-Kuznetsov equation and prove that it is locally well-posed in $H^s(\RT)$ for all $s > \frac{1}{2}$.
\end{abstract}

\maketitle

\section{Introduction}

For $s \in \R$ and initial data $u_0 \in H^s(\RT)$, we consider the Cauchy problem

\begin{equation} \tag{CP} \label{CP}
\partial_t u + \partial_x \Delta_{xy}u = \pm \partial_x \left( u^4 \right), \qquad u(t=0) = u_0,
\end{equation}

associated with the quartic Zakharov-Kuznetsov equation ($3$-gZK). The equation itself constitutes a generalization of other well-known equations in two distinct ways. On the one hand, it can be viewed as a two-dimensional analogue of the quartic Korteweg-de Vries equation, for which the well-posedness theory has been extensively studied in the past; see, for instance, \cite{Kenig1993WellposednessAS, Grünrock2001ABA, GrünrockGWPgKdV3, TaogKdV3, BourgainKDV, StaffilanigKdV, IteamgKdV}, as well as the more recent global well-posedness result at low regularity in \cite{Simao2024}.
On the other hand, $3$-gZK is a higher order generalization of the Zakharov-Kuznetsov equation (with quadratic nonlinearity), which has been derived on both $\R^3$ and $\R^2$ as a model for the propagation of nonlinear waves in plasmas \cite{Zakharov1974, Laedke1982}. \\
We briefly review the development of the well-posedness theory for $3$-gZK, both in the non-periodic and semiperiodic settings: On $\R^2$, Linares and Pastor \cite{LinaresPastorqZK2011} initiated the study by proving local well-posedness in $H^s(\R^2)$ for $s > \frac{3}{4}$, global well-posedness for $s \geq 1$ under an additional smallness assumption on the $H^1$-norm of the initial data, and ill-posedness for initial data in $\dot{H}^{\frac{1}{3}}(\R^2)$, in the sense that the data-to-solution map fails to be uniformly continuous at this scaling-critical regularity.
By further optimizing the interplay between the well-known Kato smoothing effect and two maximal function estimates, Ribaud and Vento \cite{RibaudVento2012} improved the local result to $s>\frac{5}{12}$, and Grünrock \cite{Grünrock2015OnTG} subsequently closed the remaining gap of $\frac{1}{12}$ of a derivative to the scaling threshold by proving local and small data global well-posedness in homogeneous Besov spaces $\dot{B}^\frac{1}{3}_{2,q}$, using a newly developed maximal function estimate within the framework of the function spaces $U^p$ and $V^p$. \\
Regarding the semiperiodic case, Farah and Molinet \cite{FarahMolinetgZK} established local well-posedness in $H^s(\RT)$ for all $s>1-$ by developing a linear $L^4$-estimate and using it in conjunction with a bilinear smoothing estimate due to Molinet and Pilod \cite{MolinetPilodbilin}. This result, again under a smallness assumption on the $H^1$-norm of the initial data, implied global well-posedness for all $s \geq 1$, and just recently, the local result was improved to $s> \frac{8}{15}$ \cite{JNK2025} by developing novel linear $L^4$- and $L^6$-Strichartz-type estimates and using them in the context of a fixed-point argument in Bourgain's $X_{s,b}$ spaces.   \\
In this short note, we build on the results from \cite{JNK2025} and further refine the local  theory by additionally making use of a bilinear smoothing estimate developed therein. This allows us to lower the regularity threshold to $s > \frac{1}{2}$ and the precise statement of our result is given in

\begin{theorem} \label{LWP3gZK}
The Cauchy problem $\mathrm{\eqref{CP}}$ for the quartic Zakharov-Kuznetsov equation is locally well-posed for every $s > \frac{1}{2}$. That is, for each $s > \frac{1}{2}$ and every $u_0 \in H^s(\RT)$, there exist a lifespan $\delta = \delta(\norm{u_0}_{H^{\frac{1}{2}+}}) > 0$ and a unique solution
\[ u \in X_{s,\frac{1}{2}+}^\delta(\phi) \]
to $\mathrm{\eqref{CP}}$. Moreover, for every $\tilde{\delta} \in \intoo{0,\delta}$, there exists a neighborhood \ $\Ucal \subseteq H^s(\RT)$ of $u_0$ such that the data-to-solution map
\[ S: H^s(\RT) \supseteq \Ucal \rightarrow X_{s,\frac{1}{2}+}^{\tilde{\delta}}(\phi), \qquad v_0 \mapsto S(v_0) \coloneqq v \]
is smooth.
\end{theorem}

\subsection*{Acknowledgments}
This note is intended to be part of the author's PhD thesis, and he would like to express his sincere gratitude to his supervisor, Axel Grünrock, for his continued guidance.

\section{Notation, function spaces, and Strichartz-type estimates}
 In this preliminary section, we fix notational conventions, introduce the function spaces we will work with, and recall the known linear and bilinear Strichartz estimates that will be needed in the proof of Theorem \ref{LWP3gZK}. \\ \\
In the following, we denote by $\abs{x}_2^2 \coloneqq x_1^2+...+x_n^2$ the Euclidean norm and by $\langle x \rangle \coloneqq (1+\abs{x}_2^2)^\frac{1}{2}$ the Japanese bracket of a vector $x \in \R^n$, omitting the subscript when taking the modulus of a real or complex number. Moreover, for $(\xi,q) \in \R \times \Z$, we set
$\abs{(\xi,q)}^2 \coloneqq 3\xi^2+q^2$ to simplify notation, and for $a \in \R$ and particularly small $\eps >0$, we write $a+ \coloneqq a+\eps$ and $a- \coloneqq a-\eps$ for brevity\footnote{We also often write $\infty-$ to denote a very large positive real number.}.
If $a,b$ are positive real numbers, we write $a \les b$ if there exists a constant $c>0$ such that $a\leq cb$, $a \sim b$, if both $a \les b$ and $b \les a$ hold, and $a \ll b$ if $c$ can be chosen to be very close to zero. For an admissible function $f : \R \times \R \times \T \rightarrow \C$, we denote its space-time Fourier transform by $\widehat{f} = \widehat{f}(\tau,\xi,q)$, and partial Fourier transforms are denoted using subscripts, as in $\Fcal_xf$ or $\Fcal_{xy}f$. We also use the Fourier transform and its inverse to define the Riesz and Bessel potential operators of order $-s \in \R$:
\[ I^s \coloneqq \Fcal_{xy}^{-1} \abs{(\xi,q)}_2^s \Fcal_{xy} \quad \text{and} \quad  J^s \coloneqq \Fcal_{xy}^{-1} \langle (\xi,q) \rangle^s \Fcal_{xy}. \]
As with the Fourier transform, we use subscripts whenever the operators are meant to act with respect to a single variable only. \\
We equip the Sobolev spaces $H^s(\RT)$ with the norm $\norm{\cdot}_{H^s} \coloneqq \norm{J^s \cdot}_{L^2}$, while for $1 \leq p,q \leq \infty$, we endow the mixed $L^pL^q$ spaces such as $L_{tx}^pL_y^q$, $L_t^pL_{xy}^q$, and $L_{Ty}^pL_x^q$ with the usual nested $L^p$- and $L^q$-norms (see the notation in \cite{JNK2025}). Here, we use the subscript $t$ when the integration in the time variable is taken over all of $\R$, and $T>0$ when the time variable $t$ is restricted to the interval $\intcc{-T,T}$. If, moreover, $\phi(\xi,q) \coloneqq \xi(\xi^2+q^2)$ denotes the phase function associated with the linear part of $3$-gZK, then the norms of the well-known $X_{s,b}(\phi)$ spaces, for regularity indices $s,b \in \R$, are given by
\[ \norm{f}_{X_{s,b}(\phi)}^2 \coloneqq \int_{\R^2} \sum_{q \in \Z} \langle (\xi,q) \rangle^{2s} \langle \tau - \phi(\xi,q) \rangle^{2b} \abs{\widehat{f}(\tau,\xi,q)}^2 \ \mathrm{d}(\tau,\xi). \]
Lastly, for $\delta > 0$, we also introduce the time-restricted Bourgain spaces
\[ X_{s,b}^\delta(\phi) \coloneqq \set{f|_{\intcc{-\delta,\delta} \times \R \times \T} \ : \ f \in X_{s,b}(\phi)}, \]
endowed with the norm
\[ \norm{f}_{X_{s,b}^\delta(\phi)} \coloneqq \inf \set{ \norm{\tilde{f}}_{X_{s,b}(\phi)} \ : \ \tilde{f} \in X_{s,b}(\phi), \quad \tilde{f}|_{\intcc{-\delta,\delta} \times \R \times \T} = f}, \]
as these will be the spaces in which we seek local solutions to \eqref{CP}.
For notational convenience, we suppress the dependence of the phase function, and we now collect the Strichartz-type estimates that will be needed later. \\
We begin by recalling the bilinear smoothing estimate
\begin{equation} \label{MP}
\norm{MP(u,v)}_{L_{txy}^2} \les_{\eps, b} \norm{J_y^{\frac{1}{2}+\eps}u}_{X_{0,b}} \norm{v}_{X_{0,b}}, \qquad \eps > 0, \quad b > \frac{1}{2}
\end{equation}
with
\begin{align*} &\widehat{MP(u,v)}(\tau,\xi,q) \coloneqq \int_{\R^2} \sum_{q_1 \in \Z} \abs{\abs{(\xi_1,q_1)}^2 - \abs{(\xi-\xi_1,q-q_1)}^2}^{\frac{1}{2}} \\ & \cdot \widehat{u}(\tau_1,\xi_1,q_1) \widehat{v}(\tau-\tau_1,\xi-\xi_1,q-q_1) \ \mathrm{d}(\tau_1,\xi_1),
\end{align*}
which allows for a gain of up to $\frac{1}{2}-$ derivatives for widely separated wave numbers and originally goes back to Molinet and Pilod \cite{MolinetPilodbilin} (see Proposition 3.1 in \cite{JNK2025} for a proof of this particular version). As explained in Remark 3.2 of \cite{JNK2025}, this estimate can be interpolated with a trivial bilinear bound to obtain
\begin{equation} \label{MPdual}
\norm{MP(u,v)}_{L_{txy}^2} \les \norm{u}_{X_{\frac{1}{2}+,\frac{1}{2}-}} \norm{v}_{X_{0+,\frac{1}{2}-}},
\end{equation}
which will be useful for estimates by duality.
In addition to that, we have a collection of linear estimates at our disposal, all of which were established in \cite{JNK2025}. These include the almost optimal linear $L^4$-estimate
\begin{equation} \label{L^4}
\norm{u}_{L_{txy}^4} \les_{\eps, b} \norm{u}_{X_{\eps,b}}, \qquad \eps > 0, \quad b > \frac{1}{2},
\end{equation}
the Airy $L^6$-estimate
\begin{equation} \label{AiryL^6}
\norm{I_x^\frac{1}{6}u}_{L_{txy}^6} \les_{b} \norm{J_y^{\frac{1}{3}}u}_{X_{0,b}}, \qquad b> \frac{1}{2},
\end{equation}
and the optimized $L^6$-estimate
\begin{equation} \label{optimizedL^6}
\norm{u}_{L_{Txy}^6} \les_{T, \eps, b} \norm{u}_{X_{\frac{2}{9}+\eps,b}}, \qquad T, \eps > 0, \quad b > \frac{1}{2}.
\end{equation}
For later estimates involving duality, we interpolate the above bounds with the trivial bounds
\begin{equation} \label{trivialL^2}
\norm{u}_{L_{txy}^2} = \norm{u}_{X_{0,0}}
\end{equation}
and
\begin{equation} \label{trivialLinfty}
\norm{u}_{L_{txy}^\infty} \les_{\eps, b} \norm{u}_{X_{1+\eps,b}}, \qquad \eps > 0, \quad b > \frac{1}{2},
\end{equation}
respectively, thereby obtaining the estimates
\begin{equation} \label{L^4+}
\norm{u}_{L_{txy}^{4+}} \les \norm{u}_{X_{0+,\frac{1}{2}+}}, 
\end{equation}
\begin{equation} \label{L^4-}
\norm{u}_{L_{txy}^{4-}} \les \norm{u}_{X_{0+,\frac{1}{2}-}}, 
\end{equation}
\begin{equation} \label{AiryL^6+}
\norm{I_x^\frac{1}{6}u}_{L_{txy}^{6+}} \les \norm{u}_{X_{\frac{1}{3}+,\frac{1}{2}+}}, 
\end{equation}
\begin{equation} \label{AiryL^6-}
\norm{I_x^\frac{1}{6}u}_{L_{txy}^{6-}} \les \norm{u}_{X_{\frac{1}{3}+,\frac{1}{2}-}}, 
\end{equation}
\begin{equation} \label{optimizedL^6+}
\norm{u}_{L_{Txy}^{6+}} \les \norm{u}_{X_{\frac{2}{9}+,\frac{1}{2}+}},
\end{equation}
and
\begin{equation} \label{optimizedL^6-}
\norm{u}_{L_{Txy}^{6-}} \les \norm{u}_{X_{\frac{2}{9}+,\frac{1}{2}-}}.
\end{equation}
All of these estimates were already employed in \cite{JNK2025} to establish local well-posedness for $s > \frac{8}{15}$. In the present work, we additionally incorporate the bilinear refinement
\begin{equation} \label{bilinref}
\norm{I_x^\frac{1}{4}P^1(uv)}_{L_{txy}^2} \les_{\eps, b} \norm{u}_{X_{\eps,b}} \norm{v}_{X_{\eps,b}}, \qquad \eps > 0, \quad b> \frac{1}{2}
\end{equation}
of the linear estimate \eqref{L^4}, with
\[ P^1 \coloneqq \Fcal_{xy}^{-1} \chi_{\set{\abs{3\xi^2-q^2} \gtrsim \abs{\xi}, \ \abs{\xi} \gtrsim 1}} \Fcal_{xy}, \]
(see Proposition 3.11 in \cite{JNK2025}) into the forthcoming case-by-case analysis. This additional tool will prove useful because, in the case where it fails to apply, the resulting frequency regime will be such that the Airy $L^6$-estimate \eqref{AiryL^6} can be applied a total of three times in its optimal configuration (i.e., with a loss of only $\frac{1}{6}$ of a derivative each time), ultimately leading to $s > \frac{1}{2}$.

\section{Proof of Theorem \ref{LWP3gZK}}

We now turn to the proof of a quadrilinear $X_{s,b}$-estimate. The well-posedness statement in Theorem \ref{LWP3gZK} then follows from this estimate via a fixed-point iteration, which is standard in the literature and thus omitted.

\begin{prop}
For every $s > \frac{1}{2}$, there exists $\epsilon > 0$ such that the quadrilinear estimate
\begin{equation} \label{quadrilinestimate}
\norm{\partial_x\left( \prod_{i=1}^{4}u_i \right)}_{X_{s,-\frac{1}{2}+2\epsilon}} \les \prod_{i=1}^{4} \norm{u_i}_{X_{s,\frac{1}{2}+\eps}}
\end{equation}
holds for all time-localized functions $u_1,u_2,u_3,u_4 \in X_{s,\frac{1}{2}+\eps}$.
\end{prop}

\begin{proof}
Let $\ast$ denote the convolution constraint $(\tau_0,\xi_0,q_0) \coloneqq (\tau,\xi,q) = (\tau_1+\tau_2+\tau_3+\tau_4,\xi_1+\xi_2+\xi_3+\xi_4,q_1+q_2+q_3+q_4)$. Then, by duality and after applying Parseval's identity, we can write
\begin{align*}
&\norm{\partial_x \left(\prod_{i=1}^{4} u_i \right)}_{X_{s,-\frac{1}{2}+2\eps}} \sim \sup_{\norm{f}_{X_{0,\frac{1}{2}-2\eps}} \leq 1} \left| \int_{\R^2} \sum_{q \in \Z} \int_{\R^6} \sum_{\substack{q_1,q_2,q_3 \in \Z \\ \ast}} \xi \langle (\xi,q) \rangle^s \overline{\widehat{f}}(\tau,\xi,q) \right. \\ & \left. \cdot \prod_{i=1}^{4} \widehat{u}_i(\tau_i,\xi_i,q_i) \ \mathrm{d}(\tau_1,\tau_2,\tau_3,\xi_1,\xi_2,\xi_3) \mathrm{d}(\tau,\xi) \right| \eqqcolon \sup_{\norm{f}_{X_{0,\frac{1}{2}-2\eps}} \leq 1} I_f,
\end{align*}
and since the norms appearing on the right-hand side of estimate \eqref{quadrilinestimate} depend only on the modulus of the Fourier transform, we may, without loss of generality, assume that $\widehat{f}, \widehat{u}_i \geq 0$ for all $i \in \set{1,2,3,4}$. Moreover, in the forthcoming discussion, we may restrict ourselves, by symmetry, to the frequency configuration $\abs{(\xi_1,q_1)} \geq \abs{(\xi_2,q_2)} \geq \abs{(\xi_3,q_3)} \geq \abs{(\xi_4,q_4)}$, and we finally fix the notation $\widetilde{u}(t,x,y) = u(-t,-x,-y)$ and note that the $X_{s,b}$-norms under consideration are reflection-invariant, since the phase function is odd. With this, we conclude the preliminary remarks and turn to a detailed case-by-case analysis. \\ \\
(i) \underline{$\abs{(\xi_1,q_1)} \les 1$}: \\
In this case, the pointwise estimate
\[ \abs{\xi_0} \langle (\xi_0,q_0) \rangle^s \les_s \prod_{i=1}^{4} \langle (\xi_i,q_i) \rangle^s \]
holds for every $s \in \R$, so that undoing Plancherel, followed by an application of Hölder's inequality, leads us to
\begin{align*}
I_f &\les \norm{f}_{L_{txy}^{2}} \norm{J^su_1}_{L_{txy}^{4}} \norm{J^su_2}_{L_{txy}^{4}} \norm{J^su_3}_{L_{txy}^{\infty}} \norm{J^su_4}_{L_{txy}^\infty} \\ &\les \norm{f}_{X_{0,\frac{1}{2}-2\epsilon}} \prod_{i=1}^{4} \norm{u_i}_{X_{s,\frac{1}{2}+\eps}},
\end{align*}
where, in the final step, we have assumed $0 \leq \frac{1}{2}-2\eps$ and made use of \eqref{L^4} and the trivial bound \eqref{trivialLinfty}, with all occurring derivative losses being negligible due to assumption (i). \\
(ii) \underline{$\abs{(\xi_1,q_1)} \gg 1$}: \\
(ii.1) \underline{$\abs{(\xi_1,q_1)} \gg \abs{(\xi_4,q_4)}$}: \\
(ii.1.1) \underline{$\abs{(\xi_0,q_0)} \gg \abs{(\xi_3,q_3)}$}: \\
In this situation, we have
\begin{align*}
 \abs{\xi_0} \langle (\xi_0,q_0) \rangle^s &\les \abs{\abs{(\xi_1,q_1)}^2-\abs{(\xi_4,q_4)}^2}^{\frac{1}{2}} \langle (\xi_1,q_1) \rangle^s \langle (\xi_4,q_4) \rangle^{-\frac{1}{6}} \\ & \ \ \ \cdot \abs{\abs{(\xi_0,q_0)}^2-\abs{(\xi_3,q_3)}^2}^{\frac{1}{2}} \langle (\xi_0,q_0) \rangle^{0-} \langle (\xi_3,q_3) \rangle^{-\frac{1}{6}} \langle (\xi_2,q_2) \rangle^{-\frac{2}{3}+},
 \end{align*}
 and after undoing Plancherel and applying Hölder's inequality, we obtain
\[ I_f \les \norm{MP(J^su_1,J^{-\frac{1}{6}}u_4)}_{L_{txy}^2} \norm{MP(J^{0-}f,J^{-\frac{1}{6}}\widetilde{u}_3)}_{L_{txy}^2} \norm{J^{-\frac{2}{3}+}\widetilde{u}_2}_{L_{txy}^\infty}. \]
We then estimate the first two factors using \eqref{MP} and \eqref{MPdual}, respectively, and treat the last one by means of the Sobolev embedding theorem \eqref{trivialLinfty}. This yields
\begin{align*}
... &\les \norm{f}_{X_{0,\frac{1}{2}-2\eps}} \norm{u_1}_{X_{s,\frac{1}{2}+\eps}} \norm{u_2}_{X_{\frac{1}{3}+,\frac{1}{2}+\eps}} \norm{u_3}_{X_{\frac{1}{3}+,\frac{1}{2}+\eps}} \norm{u_4}_{X_{\frac{1}{3}+,\frac{1}{2}+\eps}} \\ & \les \prod_{i=1}^{4} \norm{u_i}_{X_{s,\frac{1}{2}+\eps}},
\end{align*}
and the last step works for every $s> \frac{1}{3}$, provided that $\eps > 0$ is chosen sufficiently small. \\
(ii.1.2) \underline{$\abs{(\xi_0,q_0)} \les \abs{(\xi_3,q_3)}$}: \\
The active assumptions allow us to infer the pointwise bound
\begin{align*}
\abs{\xi_0} \langle (\xi_0,q_0) \rangle^s &\les \abs{\abs{(\xi_1,q_1)}^2 - \abs{(\xi_4,q_4)}^2}^\frac{1}{2} \langle (\xi_1,q_1) \rangle^s \langle (\xi_4,q_4) \rangle^{-\frac{7}{54}} \\ & \ \ \ \cdot \abs{\xi_0}^\frac{1}{6} \langle (\xi_0,q_0) \rangle^{-\frac{1}{3}-} \langle (\xi_2,q_2) \rangle^{\frac{4}{27}+} \langle (\xi_3,q_3) \rangle^{\frac{4}{27}+},
\end{align*}
and after applying Parseval's identity, followed by Hölder's inequality, we obtain
\[ I_f \les \norm{MP(J^su_1,J^{-\frac{7}{54}}u_4)}_{L_{txy}^2} \norm{I_x^\frac{1}{6}J^{-\frac{1}{3}-}f}_{L_{txy}^{6-}} \norm{J^{\frac{4}{27}+}\widetilde{u}_2}_{L_{Txy}^{6+}} \norm{J^{\frac{4}{27}+}\widetilde{u}_3}_{L_{Txy}^{6+}}. \]
The first factor can now be treated using \eqref{MP}, the second using \eqref{AiryL^6-}, and the remaining two factors can be estimated by means of \eqref{optimizedL^6+}. This leads us to conclude
\begin{align*}
...&\les \norm{f}_{X_{0,\frac{1}{2}-2\eps}} \prod_{i=2}^{4} \norm{u_1}_{X_{s,\frac{1}{2}+\eps}} \norm{u_i}_{X_{\frac{10}{27}+,\frac{1}{2}+\eps}} \\ &\les \prod_{i=1}^{4} \norm{u_i}_{X_{s,\frac{1}{2}+\eps}},
\end{align*}
with the last step being valid for all $s > \frac{10}{27}$, provided that $\eps > 0$ is chosen small enough. \\
(ii.2) \underline{$\abs{(\xi_1,q_1)} \sim \abs{(\xi_2,q_2)} \sim \abs{(\xi_3,q_3)} \sim \abs{(\xi_4,q_4)}$}: \\
(ii.2.1) \underline{$\abs{\abs{(\xi_i,q_i)}^2-\abs{(\xi_j,q_j)}^2} \gtrsim \abs{(\xi_1,q_1)}^2$ for some tuple $(i,j) \in \set{0,...,4}^2$, $i \neq j$}: \\
The cases corresponding to tuples $(i,j)$ with $i\neq 0$ and $j \neq 0$ can be treated in exactly the same way as in case (ii.1.2), since $\langle (\xi_0,q_0) \rangle \les \langle (\xi_k,q_k) \rangle$ holds for all $k \in \set{1,2,3,4}$. It therefore remains to examine what happens, for instance, in the case $(i,j) = (0,1)$ (all remaining cases can be treated analogously): Under the active assumptions (ii.2) and (ii.2.1), we have the pointwise bound
\begin{align*}
\abs{\xi_0} \langle (\xi_0,q_0) \rangle^s &\les \abs{\abs{(\xi_1,q_1)}^2-\abs{(\xi_0,q_0)}^2}^\frac{1}{2} \langle (\xi_1,q_1) \rangle^{s-} \langle (\xi_0,q_0) \rangle^{-\frac{1}{2}-} \prod_{k=2}^{4} \langle (\xi_k,q_k) \rangle^{\frac{1}{6}+},
\end{align*}
which allows us to conclude that
\[ I_f \les \norm{MP(J^{s-}\widetilde{u}_1,J^{-\frac{1}{2}-}f)}_{L_{txy}^2} \prod_{k=2}^{4} \norm{J^{\frac{1}{6}+}u_k}_{L_{Txy}^6} \]
(Plancherel's theorem and Hölder's inequality). The first factor is now estimated using \eqref{MPdual}, while all remaining factors can be handled by means of \eqref{optimizedL^6}. We finally obtain
\begin{align*}
...&\les \norm{f}_{X_{0,\frac{1}{2}-2\eps}} \norm{u_1}_{X_{s,\frac{1}{2}+\eps}} \prod_{k=2}^{4} \norm{u_k}_{X_{\frac{7}{18}+,\frac{1}{2}+\eps}} \\ &\les \prod_{k=1}^{4} \norm{u_k}_{X_{s,\frac{1}{2}+\eps}}, 
\end{align*}
and the last step is valid for every $s > \frac{7}{18}$, provided that $\eps > 0$ is chosen sufficiently small. This completes the discussion of this subcase. \\
(ii.2.2) \underline{$\abs{\abs{(\xi_i,q_i)}^2-\abs{(\xi_j,q_j)}^2} \ll \abs{(\xi_1,q_1)}^2$ for all tuples $(i,j) \in \set{0,...,4}^2$}: \\
We now arrange the $\xi_i$ for $i \in \set{1,2,3,4}$ in ascending order of their modulus and consider two cases. \\
(ii.2.2.1) \underline{$\abs{\xi_{\text{min}}} \leq \abs{\xi_{\text{med}_1}} \leq \abs{\xi_{\text{med}_2}} \ll \abs{\xi_{\text{max}}}$}: \\
If $\abs{\xi_{\text{max}}} \ll \abs{(\xi_1,q_1)}$, then (ii.2.2) and (ii.2.2.1) would imply
\[ \abs{q_0} \sim \abs{q_1} \sim \abs{q_2} \sim \abs{q_3} \sim \abs{q_4} \sim \abs{(\xi_1,q_1)} \]
with
\[ \abs{\abs{q_i} - \abs{(\xi_1,q_1)}} \ll \abs{(\xi_1,q_1)} \quad \text{for all} \quad i \in \set{0,1,2,3,4}, \]
which would contradict $q_1+q_2+q_3+q_4=q_0$. Hence, we must have $\abs{\xi_{\text{max}}} \sim \abs{(\xi_1,q_1)}$. Without loss of generality, let $(\abs{\xi_{\text{min}}},\abs{\xi_{\text{med}_1}},\abs{\xi_{\text{med}_2}},\abs{\xi_{\text{max}}}) = (\abs{\xi_1},\abs{\xi_2},\abs{\xi_3},\abs{\xi_4})$. The active assumptions (ii.2.2) and (ii.2.2.1) then lead to the relations
\begin{equation} \label{relations1}
\begin{aligned}
\abs{q_1} \sim \abs{q_2} \sim \abs{q_3} \sim &\abs{(\xi_1,q_1)} \quad \text{with} \quad \abs{\abs{q_i} - \abs{(\xi_1,q_1)}} \ll \abs{(\xi_1,q_1)} \\ &\quad \text{for all} \quad i \in \set{1,2,3},
\end{aligned}
\end{equation}
from which it follows (note assumption (ii.2.2)) that we must also have
\begin{equation} \label{relations2}
\abs{q_0} \sim \abs{q_4} \sim \abs{(\xi_1,q_1)} \quad \text{with} \quad \abs{\abs{q_0}-\abs{q_4}} \ll \abs{(\xi_1,q_1)}
\end{equation}
(otherwise one would obtain $\abs{q_0},\abs{q_4} \ll \abs{(\xi_1,q_1)}$, which is incompatible with the relations in \eqref{relations1} and the convolution constraint $q_1+q_2+q_3+q_4=q_0$). Now, combining the relations from \eqref{relations1} and \eqref{relations2} and taking the convolution constraint into account, it follows that $q_4$ and at least one other $q_i$ for $i \in \set{1,2,3}$ must have the same sign; without loss of generality, we take $q_1$ and $q_4$. For this configuration, we then have
\[ \abs{3(\xi_1+\xi_4)^2-(q_1+q_4)^2} \sim \abs{(\xi_1,q_1)}^2 \gg \abs{\xi_1+\xi_4} \sim \abs{\xi_4} \sim \abs{(\xi_1,q_1)} \gg 1, \]
so that we may infer the pointwise bound
\begin{align*}
\abs{\xi_0} \langle (\xi_0,q_0) \rangle^s &\les \abs{\xi_1+\xi_4}^\frac{1}{4} \langle (\xi_1,q_1) \rangle^{s-} \langle (\xi_4,q_4) \rangle^{\frac{49}{108}+} \abs{\xi_0}^\frac{1}{6} \langle (\xi_0,q_0) \rangle^{-\frac{1}{3}-} \\ & \ \ \ \cdot \langle (\xi_2,q_2) \rangle^{\frac{25}{108}+} \langle (\xi_3,q_3) \rangle^{\frac{25}{108}+}.
\end{align*}
Undoing Plancherel, followed by an application of Hölder's inequality then yields
\[ I_f \les \norm{I_x^\frac{1}{4}P^1(J^{s-}u_1J^{\frac{49}{108}+}u_4)}_{L_{txy}^2} \norm{I_x^\frac{1}{6}J^{-\frac{1}{3}-}f}_{L_{txy}^{6-}} \norm{J^{\frac{25}{108}+}\widetilde{u}_2}_{L_{Txy}^{6+}} \norm{J^{\frac{25}{108}+}\widetilde{u}_3}_{L_{Txy}^{6+}}, \]
and the first factor is handled using \eqref{bilinref}, while the remaining factors can be estimated using \eqref{AiryL^6-} and \eqref{optimizedL^6+}. We thus arrive at
\begin{align*}
...&\les \norm{f}_{X_{0,\frac{1}{2}-2\eps}} \norm{u_1}_{X_{s,\frac{1}{2}+\eps}} \prod_{i=2}^{4} \norm{u_i}_{X_{\frac{49}{108}+,\frac{1}{2}+\eps}} \\ & \les \prod_{i=1}^{4} \norm{u_i}_{X_{s,\frac{1}{2}+\eps}},
\end{align*}
which is valid for every $s > \frac{49}{108}$, provided that $\eps > 0$ is chosen small enough. \\
(ii.2.2.2) \underline{$\abs{\xi_{\text{min}}} \leq \abs{\xi_{\text{med}_1}} \leq \abs{\xi_{\text{med}_2}} \sim \abs{\xi_{\text{max}}}$}: \\ 
We may again, without loss of generality, assume $\abs{\xi_{\text{med}_2}} = \abs{\xi_3}$ and $\abs{\xi_{\text{max}}} = \abs{\xi_4}$. Due to the active assumption (ii.2.2.2), we then have both $\abs{\xi_0} \les \abs{\xi_3}$ and $\abs{\xi_0} \les \abs{\xi_4}$, so that we may write
\begin{align*}
\abs{\xi_0} \langle (\xi_0,q_0) \rangle^{s} &\les \langle (\xi_1,q_1) \rangle^{s-} \langle (\xi_2,q_2) \rangle^{\frac{1}{2}+} \abs{\xi_0}^\frac{1}{6} \langle (\xi_0,q_0) \rangle^{-\frac{1}{3}-} \abs{\xi_3}^\frac{1}{6} \langle (\xi_3,q_3) \rangle^{\frac{1}{6}+} \\ & \ \ \ \cdot \abs{\xi_4}^\frac{1}{6} \langle (\xi_4,q_4) \rangle^{\frac{1}{6}+}.
\end{align*}
By undoing Parseval's identity and a subsequent application of Hölder's inequality, it follows that
\[ I_f \les \norm{J^{s-}u_1}_{L_{txy}^4} \norm{J^{\frac{1}{2}+}u_2}_{L_{txy}^4} \norm{I_x^\frac{1}{6}J^{-\frac{1}{3}-}f}_{L_{txy}^{6-}} \norm{I_x^\frac{1}{6}J^{\frac{1}{6}+}u_3}_{L_{txy}^{6+}}  \norm{I_x^\frac{1}{6}J^{\frac{1}{6}+}u_4}_{L_{txy}^{6+}}, \]
and for the first two factors we employ \eqref{L^4}, while the remaining factors are treated using \eqref{AiryL^6-} and \eqref{AiryL^6+}. We ultimately obtain
\begin{align*}
...&\les \norm{f}_{X_{0,\frac{1}{2}-2\eps}} \norm{u_1}_{X_{s,\frac{1}{2}+\eps}} \prod_{i=2}^{4} \norm{u_i}_{X_{\frac{1}{2}+,\frac{1}{2}+\eps}} \\ &\les \prod_{i=1}^{4} \norm{u_i}_{X_{s,\frac{1}{2}+\eps}},
\end{align*}
with the last step being valid for every $s > \frac{1}{2}$, provided that $\eps > 0$ is chosen small enough. This completes the proof of \eqref{quadrilinestimate}.
\end{proof}

\ew

\bibliographystyle{amsplain}

\bibliography{refsquarticZK}

\end{document}